\documentclass[sn-mathphys-num]{sn-jnl}

\usepackage{amsmath,amssymb,amsfonts,amsthm,mathtools,mathrsfs}
\usepackage[title]{appendix}
\usepackage{booktabs,enumitem,graphicx}
\usepackage{xcolor}
\usepackage{manyfoot}
\makeatletter
\newtheoremstyle{thmstyleone}{.5\baselineskip}{.5\baselineskip}%
{\itshape}{}{\bfseries}{}{ }{\thmname{#1}\thmnumber{ #2}\thmnote{ (#3)}}%
\newtheoremstyle{thmstyletwo}{.5\baselineskip}{.5\baselineskip}%
{}{}{\bfseries}{}{ }{\thmname{#1}\thmnumber{ #2}\thmnote{ (#3)}}%
\newtheoremstyle{thmstylethree}{.5\baselineskip}{.5\baselineskip}%
{}{}{\bfseries}{}{ }{\thmname{#1}\thmnumber{ #2}\thmnote{ (#3)}}%
\@ifundefined{color@artcatboxgray}{\definecolor{artcatboxgray}{cmyk}{0,0,0,0.30}}{}
\makeatother
\theoremstyle{thmstyleone}
\newtheorem{theorem}{Theorem}
\newtheorem{proposition}{Proposition}

\makeatletter
\let\c@proposition\c@theorem
\let\c@lemma\c@theorem
\let\c@corollary\c@theorem
\makeatother
\theoremstyle{thmstyletwo}

\theoremstyle{thmstylethree}
\newtheorem{definition}{Definition}
\theoremstyle{thmstyleone}
\providecommand{\R}{\mathbb{R}}
\providecommand{\C}{\mathbb{C}}

\providecommand{\E}{\mathbb{E}}
\providecommand{\Prob}{\mathbb{P}}
\providecommand{\Var}{\operatorname{Var}}
\providecommand{\Cov}{\operatorname{Cov}}
\providecommand{\KL}{\operatorname{KL}}
\providecommand{\Rpos}{\mathbb{R}_{>0}}
\providecommand{\Leh}{\mathscr{L}}
\providecommand{\dd}{\,\mathrm{d}}
\providecommand{\supp}{\operatorname{supp}}
\providecommand{\eqdef}{\mathrel{\vcentcolon=}}

\providecommand{\LL}{\mathrm{LL}}
\providecommand{\xv}{\mathbf{x}}
\providecommand{\st}{\mathrm{st}}
\providecommand{\lr}{\mathrm{lr}}
\providecommand{\logit}{\operatorname{logit}}

\begin{document}
	
	\title[Lehmer--Lambert Distribution]{Lehmer--Lambert Distribution}
	\author*[1]{\fnm{Masoud} \sur{Ataei}}\email{masoud.ataei@utoronto.ca}
	\author[2]{\fnm{Vladimir V.} \sur{Vinogradov}}\email{vinograd@ohio.edu}
	\affil[1]{\normalsize\orgdiv{Department of Mathematical and
			Computational Sciences},\\ \orgname{University of Toronto},
		\country{Canada}}
	\affil[2]{\normalsize\orgdiv{Department of Mathematics},
		\orgname{Ohio University}, \country{USA}}
	\abstract{The Lehmer transform of a positive signal is the ratio of its power
		sums at consecutive orders, read as a function of the order, and it
		is a strictly increasing analytic curve from the smallest observation
		to the largest. Composing this curve with a map whose inverse is the
		Lambert function turns it into a probability law on the real line,
		the Lehmer--Lambert distribution, which we introduce and study. Its
		distribution, density and quantile functions are closed-form, its
		random variates are exact, and the law is invariant under changes of
		gain and of power-law calibration of the signal. Its modes lie at orders at which the signal changes scale, and for
		a sample its tails are exponential with rates equal to the two
		extreme logarithmic gaps of the sample.  We derive moments, generating functions and limit theorems, and we
		investigate applications of the Lehmer--Lambert distribution in the
		analysis of electroencephalograms of patients with major depressive
		disorder.
	}
	\maketitle
	
	\section{Introduction}\label{sec:intro}
	
	This paper introduces a probability distribution built from a
	positive signal through its Lehmer transform, the ratio of the power
	sums of the signal at two consecutive orders read as a function of
	the order. The transform is a strictly increasing analytic curve
	from the smallest observation to the largest, and any increasing map
	of that range onto the unit interval turns it into a distribution
	function. We use a map whose inverse is the Lambert function, and
	the result is the Lehmer--Lambert distribution, a law on the real
	line whose mass records the orders at which the signal changes
	scale. Its distribution function, quantile function and random
	variates are closed-form, and its density is the product of an escort covariance of the
	signal and an explicit function of the Lehmer mean of the signal. The law is the same for a signal and for any
	rescaled or power-transformed version of it, its modes sit where the
	tilted signal crosses from one scale to the next, and for a sample its tails are exponential with rates equal to
	the two extreme logarithmic gaps of the sample. The law is a representation of a signal rather than a model fitted
	to it.
	
	The Lehmer transform grew out of the Lehmer means~\cite{lehmer1971},
	the ratios of consecutive power sums that contain the harmonic,
	arithmetic and contraharmonic means and belong to the class of Gini
	means~\cite{gini1938}. The reading of the family as one
	function of its order was introduced by Ataei and
	Wang~\cite{ataei2022}, who called the order the suddency moment,
	composed the curve with a map of Lambert type, and used the modes
	and the entropy of the resulting law as features that separated
	patients with major depressive disorder from healthy controls. The
	full theory of the transform is developed by Ataei~\cite{ataei2026},
	who identified the logarithm of the transform as the unit increment
	of the free energy of an exponentially tilted ensemble, the escort
	distribution of statistical mechanics~\cite{beck1993}
	and the Esscher transform of risk theory~\cite{esscher1932}; the
	slope of the transform is then an escort covariance and the slope
	of its logarithm a Jeffreys divergence between neighbouring
	ensembles. Lehmer aggregation has also served as a trainable activation in
	neural networks~\cite{ataei2025}.
	
	The construction belongs to the tradition of generating
	distributions by composing a distribution function with a monotone
	transformation, as in the transformed-transformer families of
	Alzaatreh, Lee and Famoye~\cite{alzaatreh2013} and the Lambert
	families of Goerg~\cite{goerg2011}, in which the Lambert function
	of Corless et al.~\cite{corless1996} is used to skew or heavy-tail a
	base variable. The present law differs from these in what is composed.
	The object is the Lehmer curve of the signal rather than a
	distribution function, and the map is a second Lehmer curve, that of
	a product of an exponential and a lognormal variable, whose inverse
	is the Lambert function. The distribution function is therefore a
	composition of two Lehmer curves, and the parameters of the law are
	the number of exponential--lognormal factors and their lognormal
	variance. Functions of Lambert type also arise directly as
	distribution functions, as the generalized Lambert function of
	Kreinin, Marchenko and Vinogradov~\cite{kreinin2026} shows.
	
	Section~\ref{sec:lehmer} recalls the Lehmer transform and derives
	its inverse. Section~\ref{sec:ll} defines the Lehmer--Lambert
	distribution and derives its basic properties, and
	Section~\ref{sec:properties} its shape, tails, moments, entropy and
	limit theorems. Section~\ref{sec:estimation} treats the choice of the
	parameters and the electroencephalographic application, and
	Section~\ref{sec:conclusion} concludes. Proofs are collected in the
	Appendix.
	
	\newpage
	\section{Lehmer transform}\label{sec:lehmer}
	
	For a positive signal $\xv=(x_1,\dots,x_n)$ the Lehmer
	transform is the ratio of consecutive power sums,
	\begin{equation}\label{eq:defdata}
		\Leh(s)=\frac{\sum_{i=1}^{n}x_i^{\,s}}{\sum_{i=1}^{n}x_i^{\,s-1}},
		\qquad s\in\R .
	\end{equation}
	We read $x^{s}=e^{s\log x}$ as a Boltzmann weight at inverse
	temperature $s$ on the energy $\log x$, so that the transform is the ratio of a partition function at
	consecutive orders, and every
	property below flows from that reading. We state the theory for a
	positive Borel measure $\mu$ on $\Rpos$, which covers the empirical
	measure $\sum_i\delta_{x_i}$ of a signal and an absolutely
	continuous population at once. $X$ denotes the identity variable and
	$\log$ the natural logarithm. The results recalled without proof in
	this section are proved by Ataei~\cite{ataei2026}.
	
	\begin{definition}[Transform and escorts]\label{def:Z}
		The \emph{partition function} and the \emph{free energy} of $\mu$
		are
		\begin{equation}\label{eq:defZ}
			Z_\mu(s)\eqdef\int_{\Rpos}x^{s}\dd\mu(x),
			\qquad
			K_\mu\eqdef\log Z_\mu .
		\end{equation}
		The \emph{Lehmer transform} of $\mu$ and its logarithm are
		\begin{equation}\label{eq:defL}
			\Leh_\mu(s)\eqdef\frac{Z_\mu(s)}{Z_\mu(s-1)},
			\qquad
			\Lambda_\mu\eqdef\log\Leh_\mu .
		\end{equation}
		For every $u$ with $Z_\mu(u)<\infty$ the \emph{escort measure} of
		order $u$ is the probability measure
		\begin{equation}\label{eq:escort}
			\dd\mu_u\eqdef\frac{x^{u}\dd\mu(x)}{Z_\mu(u)} .
		\end{equation}
		For a signal the escort weights are $p_i(u)=x_i^{u}/\sum_jx_j^{u}$.
	\end{definition}
	
	The order $s$ is called the \emph{suddency moment}, a name coined
	by Ataei and Wang~\cite{ataei2022}. Raising $s$ tilts the escort
	toward the largest observations and lowering it tilts the escort
	toward the smallest, so the order controls how sharply the statistic
	responds to sudden large or small excursions of the signal, and the
	whole curve records the signal as a function of that sensitivity.
	The escorts form the exponential family with natural parameter $u$,
	sufficient statistic $\log X$ and log-partition function $K_\mu$.
	For compactly supported $\mu$ all quantities are finite for every
	real $s$. The free energy is strictly convex unless $\mu$ is a point
	mass, and $\Leh_\mu$ is real-analytic. The next theorem is the
	master identity of the theory.
	
	\begin{theorem}[Master identity]\label{thm:master}
		Let $\mu$ be compactly supported in $\Rpos$ and not a point mass.
		For all real $s$,
		\begin{equation}\label{eq:master}
			\Leh_\mu(s)=\E_{\mu_{s-1}}[X]=\exp\bigl(K_\mu(s)-K_\mu(s-1)\bigr).
		\end{equation}
		The derivatives of the free energy are escort moments,
		\begin{equation}\label{eq:Kderiv}
			K_\mu'(u)=\E_{\mu_u}[\log X],
		\end{equation}
		\begin{equation}\label{eq:Kderiv2}
			K_\mu''(u)=\Var_{\mu_u}(\log X).
		\end{equation}
		The slope of the transform is an escort covariance,
		\begin{equation}\label{eq:covid}
			\Leh_\mu'(s)=\Cov_{\mu_{s-1}}\bigl(X,\log X\bigr)
			=\Leh_\mu(s)\bigl[K_\mu'(s)-K_\mu'(s-1)\bigr]>0 .
		\end{equation}
		The slope of its logarithm is a Jeffreys divergence,
		\begin{equation}\label{eq:jeffreysslope}
			\Lambda_\mu'(s)=\int_{s-1}^{s}K_\mu''(u)\dd u=J(\mu_{s-1},\mu_s),
			\qquad
			J(P,Q)=\KL(P\|Q)+\KL(Q\|P).
		\end{equation}
		If $\supp\mu\subset[m,M]$ then, for all $s$,
		\begin{equation}\label{eq:gruss}
			\Leh_\mu'(s)\le\tfrac14(M-m)\log(M/m).
		\end{equation}
	\end{theorem}
	
	The transform at order $s$ is the mean of the signal under the escort
	of order $s-1$, and its slope is a dispersion of the tilted signal.
	The slope of its logarithm measures how distinguishable the tilted
	ensembles at inverse temperatures $s-1$ and $s$ are, or equivalently
	it is the Fisher information $I(u)=K_\mu''(u)$ of the escort family
	averaged over one unit of suddency. For a signal write
	\begin{equation}\label{eq:powersums}
		P_u=\sum_ix_i^{u},
		\qquad
		Q_u=\sum_ix_i^{u}\log x_i .
	\end{equation}
	Then the transform and the covariance identity \eqref{eq:covid} read
	\begin{equation}\label{eq:powersumslope}
		\Leh(s)=\frac{P_s}{P_{s-1}},
	\end{equation}
	\begin{equation}\label{eq:powersumslope2}
		\Leh'(s)=\frac{P_s}{P_{s-1}}\Bigl[\frac{Q_s}{P_s}-\frac{Q_{s-1}}{P_{s-1}}\Bigr].
	\end{equation}
	No differentiation is needed to evaluate the slope.
	
	\begin{theorem}[Shape of the curve]\label{thm:statgen}
		Let $\mu$ be compactly supported with
		$m=\min\supp\mu<M=\max\supp\mu$. Then $\Leh_\mu:\R\to(m,M)$ is a
		strictly increasing real-analytic bijection onto the open interval.
		The values $\Leh_\mu(0)$, $\Leh_\mu(1)$ and $\Leh_\mu(2)$ are the
		harmonic, arithmetic and contraharmonic means. Suppose that $\mu$ is finitely supported, let $w_M$ be the mass of
		$M$, and let $M'$ be the largest support point below $M$, with mass
		$w_M'$. Then, as $s\to\infty$,
		\begin{equation}\label{eq:georate}
			M-\Leh_\mu(s)\sim\frac{w_M'}{w_M}\,M\Bigl(1-\frac{M'}{M}\Bigr)\Bigl(\frac{M'}{M}\Bigr)^{s-1}.
		\end{equation}
		The mirror statement holds at the minimum as $s\to-\infty$, with
		the masses of $m$ and of the smallest support point above it.
	\end{theorem}
	
	Since the curve is a strictly increasing analytic bijection, it has
	an analytic inverse, which converts a statistic back into the order
	that generates it. We call it the suddency map.
	
	\begin{proposition}[Inverse transform]\label{prop:inverse}
		Under the hypotheses of Theorem~\ref{thm:statgen} the inverse
		$\Leh_\mu^{-1}:(m,M)\to\R$ exists and is real-analytic, with
		\begin{equation}\label{eq:inversederivative}
			\frac{\dd}{\dd t}\Leh_\mu^{-1}(t)
			=\frac{1}{\Cov_{\mu_{s_t-1}}(X,\log X)},
			\qquad s_t=\Leh_\mu^{-1}(t).
		\end{equation}
		At the arithmetic node it has the expansion
		\begin{equation}\label{eq:suddencymap}
			\Leh_\mu^{-1}(t)=1+\frac{t-\Leh_\mu(1)}{\Cov_{\mu_0}(X,\log X)}
			+O\bigl((t-\Leh_\mu(1))^{2}\bigr),
		\end{equation}
		where $\mu_0=\mu/\mu(\Rpos)$. For a signal the inverse is
		computed by the Newton iteration
		\begin{equation}\label{eq:newton}
			s_{k+1}=s_k-\frac{\Leh(s_k)-t}{\Leh'(s_k)},
		\end{equation}
		with $\Leh$ and $\Leh'$ given by the power-sum formulas
		\eqref{eq:powersumslope} and \eqref{eq:powersumslope2}. The iteration converges quadratically near the root. Since the
		curve is flat in its tails, the iteration must be safeguarded by
		bisection on a bracketing interval, which the monotonicity of the
		curve makes trivial to find. For two atoms $m<M$ the inverse is
		a logit,
		\begin{equation}\label{eq:twopointinverse}
			\Leh_\mu^{-1}(t)=1+\frac{1}{\log(M/m)}\log\frac{t-m}{M-t},
			\qquad m<t<M .
		\end{equation}
	\end{proposition}
	
	\begin{proof}
		By Theorem~\ref{thm:master} the derivative of $\Leh_\mu$ is
		positive, and by Theorem~\ref{thm:statgen} the curve is a
		real-analytic bijection onto $(m,M)$. The analytic inverse function
		theorem gives a real-analytic inverse, and the derivative of the
		inverse is the reciprocal of the covariance \eqref{eq:covid}. Expansion
		\eqref{eq:suddencymap} is the first Lagrange--B\"urmann coefficient
		at $t_0=\Leh_\mu(1)$, where the escort of order $0$ is the
		normalized measure. Newton's iteration for the analytic monotone equation $\Leh(s)=t$
		converges quadratically near the root.
		For two atoms, $\Leh_\mu(s)=(m^{s}+M^{s})/(m^{s-1}+M^{s-1})$, and
		solving $\Leh_\mu(s)=t$ gives $(M/m)^{s-1}=(t-m)/(M-t)$.
	\end{proof}
	
	The transform responds to transformations of the signal through an
	elementary calculus.
	
	\begin{theorem}[Calculus of the transform]\label{thm:calculus}
		Let $\mu$ be a positive measure whose partition function is finite
		at the orders concerned.
		\begin{enumerate}[label=\textup{(\roman*)},leftmargin=2.2em]
			\item The dilation $x\mapsto\lambda x$ sends $\Leh_\mu$ to
			$\lambda\Leh_\mu$.
			\item The power map $x\mapsto x^{r}$ sends $Z_\mu(s)$ to $Z_\mu(rs)$.
			\item Let $\mu\boxtimes\rho$ be the image of $\mu\otimes\rho$ under
			multiplication, which for probability laws is the law of a
			product of independent variables with laws $\mu$ and $\rho$.
			Then
			\begin{equation}\label{eq:character}
				Z_{\mu\boxtimes\rho}=Z_\mu Z_\rho,
				\qquad
				\Leh_{\mu\boxtimes\rho}=\Leh_\mu\Leh_\rho .
			\end{equation}
			If $\log X$ is infinitely divisible, the multiplicative
			convolution powers $\mu^{\boxtimes t}$ exist for $t>0$ and
			\begin{equation}\label{eq:power}
				\Leh_{\mu^{\boxtimes t}}=\Leh_\mu^{\,t}.
			\end{equation}
			\item The exponential law of mean $\theta$ has $\Leh(s)=\theta s$ for
			$s>0$.
			The lognormal law with parameters $(m_0,\sigma^{2})$ has
			\begin{equation}\label{eq:lognormalcurve}
				\Leh(s)=e^{m_0+\sigma^{2}(s-1/2)} .
			\end{equation}
		\end{enumerate}
	\end{theorem}
	
	Two consequences shape the construction of the next section. By
	(i) one may fix the lower endpoint of the range of the curve at any
	value, but not the upper one, since the ratio of the extremes is
	determined by the signal; moving the upper endpoint requires a power
	map, by (ii), and a power map also rescales the suddency axis. By
	(iii) and (iv), the curve of a product of an exponential and a
	lognormal variable is the product of an affine and a log-affine
	curve, whose inverse is the Lambert function, and this is where the
	Lambert function enters the theory.
	
	\section{Lehmer--Lambert distribution}\label{sec:ll}
	
	Two operations turn the curve into a law. A standardization of the
	signal fixes the range of the curve, and an increasing map then
	sends that range onto the unit interval. We fix the standardization
	first and let the parameters of the map vary afterwards, so that the
	sample space of the family does not move with the parameters.
	
	\begin{definition}[Standardization]\label{def:standard}
		Let $\mu$ be compactly supported and not a point mass, with
		extremes $m<M$. Write
		\begin{equation}\label{eq:ellmu}
			\ell_\mu\eqdef\log\frac{M}{m}.
		\end{equation}
		Fix a level $L_+>1$ and write $\ell\eqdef\log L_+$. The
		\emph{standardization} of $\mu$ is the pushforward $\nu$ of $\mu$
		under the power map
		\begin{equation}\label{eq:standardization}
			x\longmapsto\Bigl(\frac xm\Bigr)^{\tau},
			\qquad
			\tau\eqdef\frac{\log L_+}{\log(M/m)} .
		\end{equation}
		For a signal the standardized values are $y_i=(x_i/m)^{\tau}$.
	\end{definition}
	
	The measure $\nu$ is carried by $[1,L_+]$, with $1$ and $L_+$ its
	extremes. By Theorem~\ref{thm:statgen}, $\Leh_\nu$ is a strictly
	increasing bijection of $\R$ onto $(1,L_+)$.
	
	\begin{definition}[Lehmer--Lambert distribution]\label{def:ll}
		For $\alpha>0$ and $\beta\ge0$ put
		\begin{equation}\label{eq:G}
			G(z)\eqdef z^{1/\alpha}e^{\beta z},
			\qquad 1\le z\le L_+ ,
		\end{equation}
		\begin{equation}\label{eq:AC}
			A\eqdef G(L_+),
			\qquad
			C\eqdef\frac{1}{A-e^{\beta}} .
		\end{equation}
		The \emph{Lehmer--Lambert distribution} of $\mu$ with parameters
		$(\alpha,\beta)$, written $\LL(\alpha,\beta)$, the level $L_+$ and
		the signal $\mu$ being understood, is the law of the random suddency
		moment $S$ with distribution function
		\begin{equation}\label{eq:llF}
			F_S(s)=C\bigl(G(\Leh_\nu(s))-e^{\beta}\bigr),
			\qquad s\in\R .
		\end{equation}
		The value $\Leh_\nu(s)$ of the transform at a fixed order $s$ is the
		Lehmer mean of order $s$ of the standardized signal, and the
		variable
		\begin{equation}\label{eq:Zdef}
			Z\eqdef\Leh_\nu(S),
		\end{equation}
		the Lehmer mean taken at the random order $S$, is the \emph{random
			Lehmer mean} of the signal.
	\end{definition}
	
	The map $G$ increases on $[1,L_+]$ from $e^{\beta}$ to $A$ and the
	curve $\Leh_\nu$ increases from $1$ to $L_+$, so the right side of
	definition \eqref{eq:llF} increases continuously from $0$ to $1$ and is a
	distribution function. In explicit form,
	\begin{equation}\label{eq:llFexplicit}
		F_S(s)=\frac{\Leh_\nu(s)^{1/\alpha}e^{\beta\Leh_\nu(s)}-e^{\beta}}
		{L_+^{1/\alpha}e^{\beta L_+}-e^{\beta}} .
	\end{equation}
	The random Lehmer mean takes values in $(1,L_+)$, and its
	distribution function is $C(G(z)-e^{\beta})$ whatever the signal.
	The parameters therefore act on the random Lehmer mean, while the
	signal enters only through the inverse curve $\Leh_\nu^{-1}$ of
	Proposition~\ref{prop:inverse}. The map $G$ was chosen by
	Ataei and Wang~\cite{ataei2022} for the closed-form normalization it
	affords, but it belongs to the theory, because it is itself a
	Lehmer curve and the distribution function \eqref{eq:llF} is then a
	composition of two Lehmer curves. Other increasing maps are Lehmer
	curves too, and the composition principle applies to each of them;
	the map $G$ is the one that brings in the Lambert function.
	
	\begin{theorem}[Composition of two curves]
		\label{thm:origin}
		Let $E$ be standard exponential and $\Lambda$ lognormal with
		parameters $(\sigma^{2}/2,\sigma^{2})$, independent. Let $\rho$ be
		the law of $E\Lambda$. Then, for $s>0$,
		\begin{equation}\label{eq:rhocurve}
			\Leh_\rho(s)=s\,e^{\sigma^{2}s},
		\end{equation}
		\begin{equation}\label{eq:rhoinverse}
			\Leh_\rho^{-1}(t)=\frac{W_0(\sigma^{2}t)}{\sigma^{2}} .
		\end{equation}
		The variable $\log(E\Lambda)$ is infinitely divisible. So $\rho$
		generates the semigroup $\rho_t=\rho^{\boxtimes t}$, $t>0$, with
		\begin{equation}\label{eq:rhotcurve}
			\Leh_{\rho_t}=\Leh_\rho^{\,t}.
		\end{equation}
		With $t=1/\alpha$ and $\sigma^{2}=\alpha\beta$,
		\begin{equation}\label{eq:GisLehmer}
			G=\Leh_{\rho_t},
		\end{equation}
		\begin{equation}\label{eq:composition}
			F_S(s)=\frac{\Leh_{\rho_t}\bigl(\Leh_\nu(s)\bigr)-\Leh_{\rho_t}(1)}
			{\Leh_{\rho_t}(L_+)-\Leh_{\rho_t}(1)} .
		\end{equation}
		Let $V$ be uniform on $(\Leh_{\rho_t}(1),\Leh_{\rho_t}(L_+))$. Then
		\begin{equation}\label{eq:Srep}
			S=\Leh_\nu^{-1}\bigl(\Leh_{\rho_t}^{-1}(V)\bigr)
			\sim\LL(\alpha,\beta).
		\end{equation}
	\end{theorem}
	
	The parameters thereby acquire meaning. The number $t=1/\alpha$
	counts the exponential--lognormal factors in the reference law, and
	$\sigma^{2}=\alpha\beta$ is the lognormal variance of each factor.
	The Lambert function enters because $W_0$ inverts $se^{\sigma^{2}s}$,
	and the Omega constant $\Omega=W_0(1)$ solves
	\begin{equation}\label{eq:omegadef}
		\Omega e^{\Omega}=1 .
	\end{equation}
	It is the suddency moment at which the exponential--lognormal
	product with $\sigma^{2}=1$ has unit Lehmer mean. The curve
	$\Leh_\rho$ is defined for positive orders only, and the composition
	uses $\Leh_{\rho_t}$ on $[1,L_+]$ alone. Read as a recipe, representation \eqref{eq:composition}
	says that any law $\rho$ with a known curve gives a distribution on
	the suddency line by the same composition. The exponential reference
	gives the uniform map, the lognormal reference gives the map
	$e^{\sigma^{2}z}$, and the inverse curve of the lognormal, $\log z$,
	gives a parameter-free law that we obtain below as the logarithmic
	limit of the family. We develop the exponential--lognormal case, the natural composition
	with a Lambert quantile.
	
	\begin{theorem}[Density and quantiles]
		\label{thm:basic}
		Let $S\sim\LL(\alpha,\beta)$ and $Z=\Leh_\nu(S)$.
		\begin{enumerate}[label=\textup{(\roman*)},leftmargin=2.2em]
			\item The random Lehmer mean has density
			\begin{equation}\label{eq:lgdensity}
				f_Z(z)=C\Bigl(\frac1\alpha+\beta z\Bigr)z^{\frac1\alpha-1}e^{\beta z},
				\qquad 1<z<L_+ .
			\end{equation}
			\item The suddency moment has density
			\begin{equation}\label{eq:ll}
				f_S(s)=f_Z\bigl(\Leh_\nu(s)\bigr)\,\Cov_{\nu_{s-1}}\bigl(X,\log X\bigr).
			\end{equation}
			For a signal, with $P_u$ and $Q_u$ the power sums of
			definition \eqref{eq:powersums} for the standardized values,
			\begin{equation}\label{eq:llpower}
				\Cov_{\nu_{s-1}}\bigl(X,\log X\bigr)
				=\frac{P_s}{P_{s-1}}\Bigl[\frac{Q_s}{P_s}-\frac{Q_{s-1}}{P_{s-1}}\Bigr].
			\end{equation}
			\item For $\beta>0$ the quantile function of $Z$ is
			\begin{equation}\label{eq:Qquantile}
				Q(p)=\frac{1}{\alpha\beta}\,
				W_0\Bigl(\alpha\beta\bigl(\tfrac{p}{C}+e^{\beta}\bigr)^{\alpha}\Bigr),
				\qquad 0<p<1 ,
			\end{equation}
			and at $\beta=0$ it is $Q(p)=(p/C+1)^{\alpha}$.
			The quantile function of $S$ is
			\begin{equation}\label{eq:quantile}
				F_S^{-1}(p)=\Leh_\nu^{-1}\bigl(Q(p)\bigr).
			\end{equation}
			\item The variable $G(Z)$ is uniform on $[e^{\beta},e^{\beta}+C^{-1}]$.
			Conversely, let $U$ be uniform on $(0,1)$. Then
			\begin{equation}\label{eq:sampling}
				S=\Leh_\nu^{-1}\bigl(Q(U)\bigr)\sim\LL(\alpha,\beta).
			\end{equation}
			\item The survival function and the hazard rate are
			\begin{equation}\label{eq:survival}
				1-F_S(s)=C\bigl(A-G(\Leh_\nu(s))\bigr),
				\qquad
				h_S(s)=\frac{f_S(s)}{C\bigl(A-G(\Leh_\nu(s))\bigr)} .
			\end{equation}
		\end{enumerate}
	\end{theorem}
	
	The density formula \eqref{eq:ll} has two factors. The first depends on the
	parameters and on the random Lehmer mean only, while the second is
	the escort covariance of the standardized signal, which is where the
	structure of the signal enters. Sampling by the representation \eqref{eq:sampling} is exact, the
	only numerical step being the inversion of the monotone analytic
	curve by the Newton iteration \eqref{eq:newton}.
	
	\begin{proposition}[Invariance]
		\label{prop:invariance}
		Let $\pi$ be the \emph{log-profile} of $\mu$, the pushforward of
		$\mu/\mu(\Rpos)$ under
		\begin{equation}\label{eq:theta}
			\theta(x)=\frac{\log x-\log m}{\ell_\mu}\in[0,1].
		\end{equation}
		Let $K_\pi(t)=\log\int_0^1e^{t\theta}\dd\pi(\theta)$ be its cumulant
		generating function. Then
		\begin{equation}\label{eq:profilecurve}
			\Lambda_\nu(s)=K_\pi(\ell s)-K_\pi(\ell s-\ell),
			\qquad
			y_i=L_+^{\theta_i}.
		\end{equation}
		The Lehmer--Lambert distribution depends on $\mu$ only through its
		log-profile. It is invariant under
		\begin{equation}\label{eq:gaugegroup}
			x\longmapsto c\,x^{r},\qquad c,r>0,
		\end{equation}
		and the log-profile is a maximal invariant of that action. Under
		$x\mapsto1/x$ the log-profile is reflected and
		\begin{equation}\label{eq:reflection}
			\Leh_{\check\nu}(s)=\frac{L_+}{\Leh_\nu(1-s)} .
		\end{equation}
	\end{proposition}
	
	The maps \eqref{eq:gaugegroup} are the changes of gain and of
	power-law calibration of the signal. Representation
	\eqref{eq:profilecurve} places the construction on the cumulant
	generating function of a probability measure on the unit interval,
	where the level $L_+$ is the length of the increment in the natural
	parameter and the random Lehmer mean is the exponential of a random
	increment of $K_\pi$.
	
	Several members and limits of the family are of independent
	interest. Two of them involve the extreme logarithmic gaps of the
	standardized signal,
	\begin{equation}\label{eq:gaps}
		r_+\eqdef\tau\log\frac{x_{(n)}}{x_{(n-1)}},
		\qquad
		r_-\eqdef\tau\log\frac{x_{(2)}}{x_{(1)}},
	\end{equation}
	the order statistics being taken among distinct values.
	
	\begin{proposition}[Special cases and limits]
		\label{prop:limits}\leavevmode
		\begin{enumerate}[label=\textup{(\roman*)},leftmargin=2.2em]
			\item At $\beta=0$,
			\begin{equation}\label{eq:powermap}
				F_S(s)=\frac{\Leh_\nu(s)^{1/\alpha}-1}{L_+^{1/\alpha}-1},
			\end{equation}
			and $Z^{1/\alpha}$ is uniform on $(1,L_+^{1/\alpha})$. At
			$\alpha=1$, $\beta=0$ the random Lehmer mean is uniform.
			\item As $\alpha\to\infty$ at $\beta=0$,
			\begin{equation}\label{eq:jeffreysF}
				F_S(s)\longrightarrow\frac{\log\Leh_\nu(s)}{\log L_+} .
			\end{equation}
			This is the \emph{Jeffreys law} of $\nu$, named after the
			Jeffreys divergence that appears in its density in the next
			section and not after the Jeffreys prior of the escort family,
			which is proportional to the square root of the Fisher
			information. Applied to $\mu$ without standardization it is the
			parameter-free law
			\begin{equation}\label{eq:jeffreysraw}
				F_S(s)=\frac{\Lambda_\mu(s)-\log m}{\ell_\mu} .
			\end{equation}
			Under it $\log\Leh_\mu(S)$ is uniform on $(\log m,\log M)$ and
			\begin{equation}\label{eq:logmean}
				\E\bigl[\Leh_\mu(S)\bigr]=\frac{M-m}{\ell_\mu},
			\end{equation}
			the logarithmic mean of the extremes.
			\item Let $L_+=1+\Omega$ and $\Leh_\rho(y)=ye^{y}$, the reference curve
			with $\sigma^{2}=1$. Then
			\begin{equation}\label{eq:canonical}
				F_S(s)=\Leh_\rho\bigl(\Leh_\nu(s)-1\bigr)
			\end{equation}
			is a distribution function, the \emph{canonical Omega law}, and
			$Y=\Leh_\nu(S)-1=W_0(U)$ with $U$ uniform on $(0,1)$.
			\item Let $\mu$ be finitely supported, let $w_M,w_M'$ be the masses
			of the largest support point and of the largest support point
			below it, and let $r_+$ be the upper gap of definition
			\eqref{eq:gaps}. Put
			\begin{equation}\label{eq:lambdaa}
				\lambda=\beta+\frac{1}{\alpha L_+},
				\qquad
				a_+=L_+\bigl(e^{r_+}-1\bigr)\frac{w_M'}{w_M} .
			\end{equation}
			As $\lambda\to\infty$,
			\begin{equation}\label{eq:gumbel}
				\lambda(L_+-Z)\xrightarrow{d}\mathrm{Exp}(1),
			\end{equation}
			\begin{equation}\label{eq:gumbel2}
				r_+S-\log(a_+\lambda)\xrightarrow{d}\mathrm{Gumbel}.
			\end{equation}
			This covers $\beta\to\infty$ and $\alpha\downarrow0$ alike.
			\item As $L_+\downarrow1$ at fixed $(\alpha,\beta)$, $\tau S$
			converges in law to the variable $T$ with density
			\begin{equation}\label{eq:infoprofile}
				f_T(u)=\frac{K_\mu''(u)}{\ell_\mu} .
			\end{equation}
			This is the Fisher information of the escort family of the raw
			signal, read as a probability density.
		\end{enumerate}
	\end{proposition}
	
	The canonical Omega law of (iii) needs no normalizing constant. It
	is not a member of the family, whose map at $\alpha=\beta=1$ is
	$ze^{z}$ rather than $(z-1)e^{z-1}$, but it comes from the
	composition principle of Theorem~\ref{thm:origin} with the argument
	of the reference curve translated by one, and the moments of $Y$
	are polynomials in $\Omega^{\pm1}$. The constraint $\alpha\le1$ is
	not needed for the definition but carries the shape. It is the
	condition under which the density of the random Lehmer mean is
	increasing for every $\beta\ge0$ and every level and under which
	the family is ordered in $\alpha$, as shown in the next section,
	and by Theorem~\ref{thm:origin} it is the condition $t\ge1$ that
	the reference law contain at least one whole exponential--lognormal
	factor. The shape and ordering results of the next section are
	stated under it. The logarithmic limit (ii) lies at
	the other end of the parameter range and has the richest structure,
	to which we return in the next section.
	
	\begin{figure}[!t]
		\centering
		\includegraphics[width=.85\textwidth]{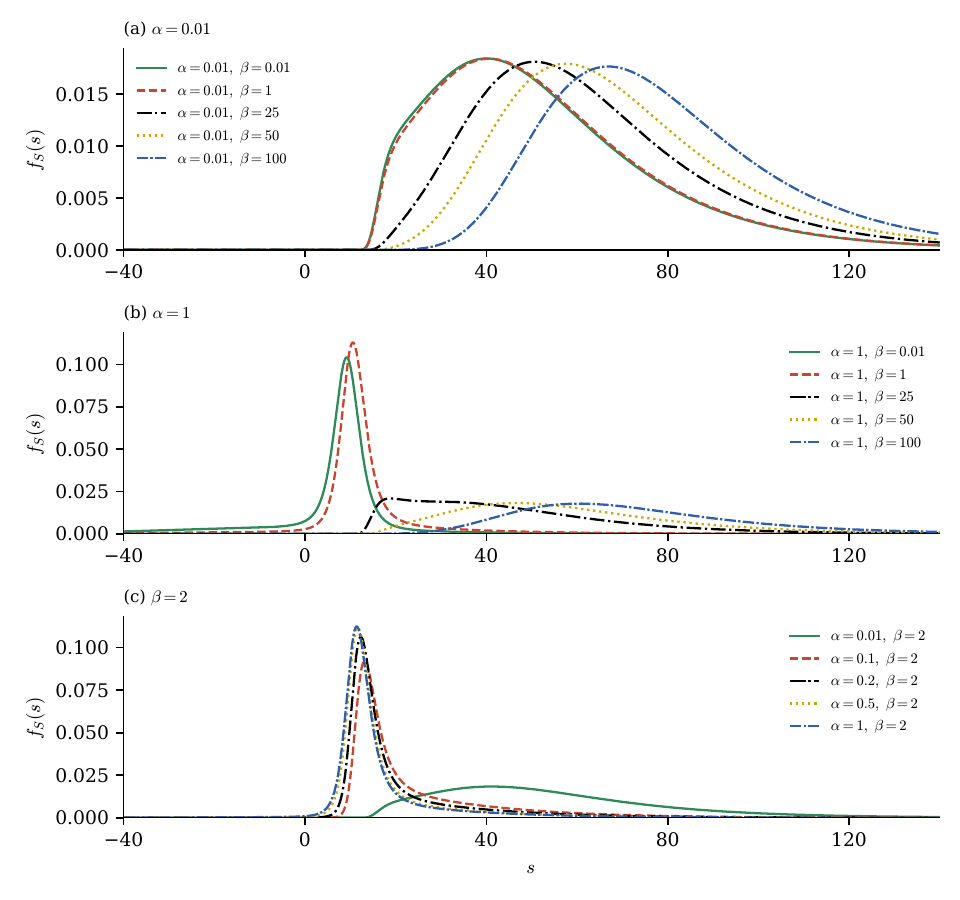}
		\caption{Densities of one signal at $L_+=e$. (a) $\alpha=0.01$ and
			$\beta$ varying. (b) $\alpha=1$ and $\beta$ varying. (c) $\beta=2$
			and $\alpha$ varying.}
		\label{fig:1}
	\end{figure}
	
	Figure~\ref{fig:1} shows the densities of one signal at the level
	$L_+=e$ for several parameter pairs. Increasing $\beta$ or
	decreasing $\alpha$ moves mass toward higher suddency and spreads
	the law, while the sharp peak in (b) and (c), which comes from a concavity
	episode of the curve, stays in place, because the shape of the
	law is carried by the escort covariance of the signal and the
	parameters act only through the density of the random Lehmer mean.
	The two parameters act largely through the combination
	$\lambda=\beta+1/(\alpha L_+)$ of the concentration limit in
	Proposition~\ref{prop:limits}(iv). When $\alpha$ is small the term
	$1/(\alpha L_+)$ dominates and moderate values of $\beta$ change
	little, as in (a); when $\alpha$ is of order one, $\beta$ governs the
	law, as in (b); and at fixed $\beta$ the law is insensitive to
	$\alpha$ until $\alpha$ becomes very small, as in (c).
	
	\section{Properties}\label{sec:properties}
	
	Throughout this section $S\sim\LL(\alpha,\beta)$ and
	$Z=\Leh_\nu(S)$. For the tail, extreme-value and generating-function results we
	assume that $\mu$ is finitely supported, that is, a finite sum of
	point masses, which is the case of a sample with multiplicities; the
	standardization preserves the masses. We then write $w_M,w_M'$ for
	the masses of the largest support point and of the largest support
	point below it, $w_m,w_m'$ for the masses of the smallest support
	point and of the smallest support point above it, and $r_\pm$ for the
	extreme logarithmic gaps of definition \eqref{eq:gaps}. For a sample
	with distinct extremes all four masses equal one. The
	log-derivative of the density of the random Lehmer mean is
	\begin{equation}\label{eq:psi}
		\psi(z)\eqdef\frac{\dd}{\dd z}\log f_Z(z)
		=\frac{\alpha\beta}{1+\alpha\beta z}+\frac{1-\alpha}{\alpha z}+\beta .
	\end{equation}
	
	\begin{theorem}[Shape]\label{thm:shape}\leavevmode
		\begin{enumerate}[label=\textup{(\roman*)},leftmargin=2.2em]
			\item For $\alpha\le1$ the density $f_Z$ is nondecreasing and
			log-concave and $G$ is convex, whatever $\beta\ge0$ and $L_+$,
			strictly so unless $\alpha=1$ and $\beta=0$. For $\alpha>1$ this
			fails when $\beta$ is small.
			\item For $\alpha\le1$ every interior mode of $S$ satisfies
			\begin{equation}\label{eq:modeeq}
				\frac{\Leh_\nu''(s)}{\Leh_\nu'(s)}=-\Leh_\nu'(s)\,\psi\bigl(\Leh_\nu(s)\bigr)\le0 ,
			\end{equation}
			with equality only when $\alpha=1$ and $\beta=0$, so the modes
			lie in the closed concavity episodes of the curve.
			\item The family increases in the likelihood-ratio order as $\beta$
			increases, and as $\alpha$ decreases within the range
			$\alpha\le1$. Every member dominates the Jeffreys law $S_J$ of
			$\nu$,
			\begin{equation}\label{eq:lrjeffreys}
				S\ge_{\lr}S_J .
			\end{equation}
			For a sample with distinct extremes, $\E[S]\ge\tfrac12$.
			\item For two standardized measures $\nu,\nu'$ at the same level,
			\begin{equation}\label{eq:kolmogorov}
				\sup_s\bigl|F_S(s)-F_{S'}(s)\bigr|
				\le\Bigl(\sup_{[1,L_+]}f_Z\Bigr)\sup_s\bigl|\Leh_\nu(s)-\Leh_{\nu'}(s)\bigr| ,
			\end{equation}
			and $\sup f_Z=CG'(L_+)$ when $\alpha\le1$.
		\end{enumerate}
	\end{theorem}
	
	The concavity episodes of the curve are the orders at which the
	escort has just crossed from one scale of the signal to the next, so
	every mode of the law lies where the tilted signal changes scale.
	The converse fails, because the map $G$ can split one episode into
	several modes, and whether it does depends on the parameters through
	the combination $\lambda=\beta+1/(\alpha L_+)$ of
	Proposition~\ref{prop:limits}(iv). When $\lambda$ is small the
	density of the random Lehmer mean is nearly flat, the density of $S$
	follows the escort covariance, and there is one mode per episode.
	When $\lambda$ is large the mass concentrates at the upper end of
	the last episode, as in the concentration limit of
	Proposition~\ref{prop:limits}(iv), and the law is unimodal. At
	intermediate values of $\lambda$ the two terms of the
	log-derivative $\psi(\Leh_\nu)\Leh_\nu'+\Leh_\nu''/\Leh_\nu'$ of the
	density are comparable, and since the first is not monotone they
	can cross more than once inside one episode, so a signal with a
	single concavity episode may have a bimodal law. For the Jeffreys
	law the stationary points are located exactly, as shown later in
	this section. For $\alpha\le1$ the random Lehmer mean has an
	increasing hazard rate, unbounded at $L_+$. The likelihood-ratio
	order implies the usual stochastic order, so every quantile of $S$,
	and $\E[S]$ when finite, increases in $\beta$ and, for $\alpha\le1$,
	decreases in $\alpha$. The restriction on $\alpha$ is needed, since
	for two values of $\alpha$ above one and $\beta$ small the order can
	reverse. The bound $\E[S]\ge\tfrac12$ is inherited from the Jeffreys
	law, whose mean is one half for every sample without ties, as shown
	later in this section.
	
	\begin{theorem}[Tails and generating function]
		\label{thm:tails}
		Let $\mu$ be finitely supported. Put
		\begin{equation}\label{eq:tailconstants}
			c_+\eqdef f_Z(L_+)\,L_+\bigl(e^{r_+}-1\bigr)\frac{w_M'}{w_M},
			\qquad
			c_-\eqdef f_Z(1)\bigl(1-e^{-r_-}\bigr)\frac{w_m'}{w_m}.
		\end{equation}
		\begin{enumerate}[label=\textup{(\roman*)},leftmargin=2.2em]
			\item As $s\to+\infty$, for some $\delta>0$,
			\begin{equation}\label{eq:tails}
				1-F_S(s)=c_+e^{-r_+s}\bigl(1+O(e^{-\delta s})\bigr),
			\end{equation}
			\begin{equation}\label{eq:tails2}
				F_S(-s)=c_-e^{-r_-s}\bigl(1+O(e^{-\delta s})\bigr).
			\end{equation}
			The density satisfies $f_S(s)\sim r_+c_+e^{-r_+s}$ and the hazard
			rate tends to $r_+$. As $p\uparrow1$,
			\begin{equation}\label{eq:tailquantile}
				F_S^{-1}(p)=\frac{1}{r_+}\log\frac{c_+}{1-p}+o(1).
			\end{equation}
			\item The moment generating function $M_S(t)=\E[e^{tS}]$ is finite
			exactly for $-r_-<t<r_+$. The cumulant function $K_S=\log M_S$
			is strictly convex and steep at both ends, and
			\begin{equation}\label{eq:abelian}
				\lim_{t\uparrow r_+}(r_+-t)\,M_S(t)=r_+c_+,
			\end{equation}
			\begin{equation}\label{eq:abelian2}
				\lim_{t\downarrow-r_-}(t+r_-)\,M_S(t)=r_-c_- .
			\end{equation}
			
		\end{enumerate}
	\end{theorem}
	
	The tails are decided by the two extreme gaps of the signal and by
	nothing else in its interior. Suddency mass far from the origin
	certifies isolated extremes, the rates $r_\pm$ are scale-free
	measures of that isolation, and near-ties at an extreme make the
	corresponding rate small and the tail long. The Abelian relations \eqref{eq:abelian} and \eqref{eq:abelian2}
	are the counterpart, for
	the moment generating function of the suddency moment, of the
	Abelian tail law that Ataei~\cite{ataei2026} proved for the transform
	at a power-law tail index.
	
	\begin{figure}[t]
		\centering
		\includegraphics[width=\textwidth]{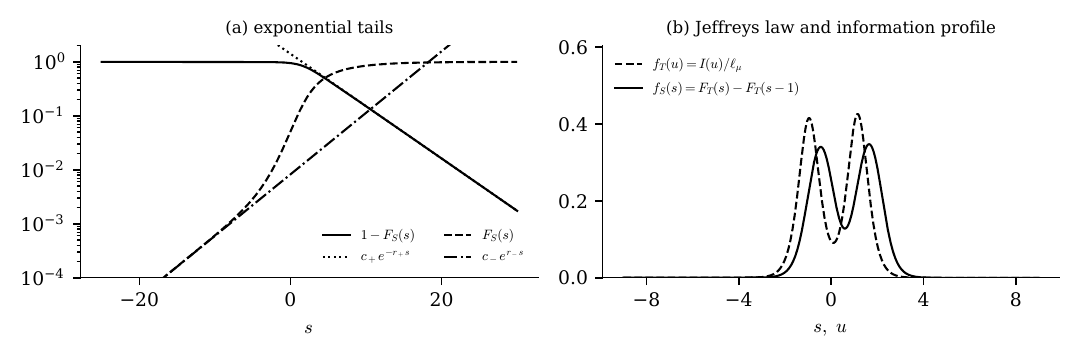}
		\caption{(a) Tails of the law of the sample $(1,1.3,1.9,2.4,3)$ at
			$(\alpha,\beta,L_+)=(0.6,1.3,3)$ and the asymptotes of the tail
			relations \eqref{eq:tails} and \eqref{eq:tails2}. (b) Jeffreys law and information profile of a
			sample on three scales.}
		\label{fig:2}
	\end{figure}
	
	Moments of the suddency moment are not elementary, since the inverse
	curve is not, but the moments of the random Lehmer mean are
	confluent hypergeometric and obey a recurrence. For $a>0$ and $\gamma\in\C$ write
	\begin{equation}\label{eq:Mdef}
		M_\gamma(a)\eqdef\int_1^{L_+}z^{a-1}e^{\gamma z}\dd z
		=\frac{L_+^{a}}{a}\,{}_1F_1(a;a+1;\gamma L_+)-\frac1a\,{}_1F_1(a;a+1;\gamma),
	\end{equation}
	with ${}_1F_1$ Kummer's function. Put
	\begin{equation}\label{eq:aEk}
		a\eqdef\frac1\alpha,
		\qquad
		H_k\eqdef C\bigl(L_+^{k}A-e^{\beta}\bigr),
	\end{equation}
	so that $H_0=1$.
	
	\begin{theorem}[Moments and entropy]
		\label{thm:moments}\leavevmode
		\begin{enumerate}[label=\textup{(\roman*)},leftmargin=2.2em]
			\item For every complex $w$ and $t$,
			\begin{equation}\label{eq:lgmoments}
				\E[Z^{w}]=C\bigl(L_+^{w}A-e^{\beta}\bigr)-wC\,M_\beta(w+a),
			\end{equation}
			\begin{equation}\label{eq:lgmgf}
				\E\bigl[e^{tZ}\bigr]=C\bigl(e^{tL_+}A-e^{t+\beta}\bigr)-tC\,M_{\beta+t}(a+1).
			\end{equation}
			Both are entire functions.
			\item For $\beta>0$ the moments $m_k=\E[Z^{k}]$ satisfy
			\begin{equation}\label{eq:m1}
				m_1=H_1-\frac1\beta+\frac{a}{\beta}\,C\,M_\beta(a),
			\end{equation}
			and, for $k\ge2$,
			\begin{equation}\label{eq:momentrec}
				m_k=H_k+\frac{k\,a}{\beta(k-1)}\,H_{k-1}
				-\frac{k\,(k-1+a)}{\beta(k-1)}\,m_{k-1}.
			\end{equation}
			Every moment and cumulant is an explicit function of
			$(\alpha,\beta,L_+)$ and of the single value $M_\beta(1/\alpha)$.
			
			\item The differential entropies are
			\begin{equation}\label{eq:entropyZ}
				h(Z)=-\log C-\E\log\Bigl(\frac1\alpha+\beta Z\Bigr)
				-\Bigl(\frac1\alpha-1\Bigr)\E[\log Z]-\beta\E[Z],
			\end{equation}
			\begin{equation}\label{eq:entropy}
				h(S)=h(Z)-\E\bigl[\log\Cov_{\nu_{S-1}}(X,\log X)\bigr].
			\end{equation}
		\end{enumerate}
	\end{theorem}
	
	The entropy of $S$ is a signal-free constant minus the mean
	logarithmic escort covariance of the signal, which measures the
	average log-dispersion of the signal under tilting. Both $\E[S]$ and
	$h(S)$ are integrals against the density of the random
	Lehmer mean over the finite interval $(1,L_+)$, by the change of variables $z=\Leh_\nu(s)$,
	\begin{equation}\label{eq:featuresZ}
		\E[S]=\int_1^{L_+}\Leh_\nu^{-1}(z)f_Z(z)\dd z,
	\end{equation}
	\begin{equation}\label{eq:featuresZ2}
		\E\bigl[\log\Cov_{\nu_{S-1}}(X,\log X)\bigr]
		=\int_1^{L_+}\log\Leh_\nu'\bigl(\Leh_\nu^{-1}(z)\bigr)f_Z(z)\dd z .
	\end{equation}
	Both integrands have integrable logarithmic singularities at the
	endpoints, and both are computed by quadrature with the inverse of
	Proposition~\ref{prop:inverse}.
	
	For two atoms $m<M$ the inverse curve is the logit of formula
	\eqref{eq:twopointinverse} and the law closes completely. With
	$\varsigma(x)=(1+e^{-x})^{-1}$ the standardized curve is
	$\Leh_\nu(s)=1+(L_+-1)\varsigma(\ell(s-1))$, so that
	\begin{equation}\label{eq:logit}
		S=1+\frac1\ell\logit\Bigl(\frac{Z-1}{L_+-1}\Bigr),
	\end{equation}
	and by the density formula \eqref{eq:ll} $S$ is the logistic law of
	location $1$ and scale $1/\ell$ tilted by the likelihood ratio
	$(L_+-1)f_Z(1+(L_+-1)\varsigma)$, which is increasing for $\alpha\le1$;
	the law is exactly logistic at $\alpha=1$, $\beta=0$. This is the
	benchmark member, as the logistic law is the benchmark information
	profile.
	
	The logarithmic member of Proposition~\ref{prop:limits}(ii)
	deserves a theorem of its own. Everything about it is expressed
	through the information geometry of the escort family, and it
	explains the shape of every member of the family. Let
	\begin{equation}\label{eq:fisherinfo}
		I(u)=K_\mu''(u)=\Var_{\mu_u}(\log X)
	\end{equation}
	be the Fisher information of the escort family of the raw signal.
	
	\begin{theorem}[Jeffreys law]
		\label{thm:jeffreys}
		Let $S$ have the Jeffreys law of $\mu$, with distribution
		function \eqref{eq:jeffreysraw}.
		\begin{enumerate}[label=\textup{(\roman*)},leftmargin=2.2em]
			\item The density is a normalized Jeffreys divergence between
			consecutive escorts,
			\begin{equation}\label{eq:jeffreysdensity}
				f_S(s)=\frac{J(\mu_{s-1},\mu_s)}{\ell_\mu}
				=\frac1{\ell_\mu}\int_{s-1}^{s}I(u)\dd u,
				\qquad
				\int_\R I(u)\dd u=\ell_\mu .
			\end{equation}
			\item $S=T+U$ with $T$ and $U$ independent, $U$ uniform on $(0,1)$
			and $T$ the \emph{information profile} of the signal,
			\begin{equation}\label{eq:Tlaw}
				f_T(u)=\frac{I(u)}{\ell_\mu},
			\end{equation}
			\begin{equation}\label{eq:Tlaw2}
				F_T(u)=\frac{K_\mu'(u)-\log m}{\ell_\mu} .
			\end{equation}
			Hence
			\begin{equation}\label{eq:unitincrement}
				f_S(s)=F_T(s)-F_T(s-1)\le1,
			\end{equation}
			and the stationary points of the density of $S$, among them its
			modes, are the points where $I(s)=I(s-1)$. The
			cumulants satisfy $\kappa_j(S)=\kappa_j(T)+B_j/j$, with $B_j$ the
			Bernoulli numbers and $B_1=\tfrac12$.
			\item If $\mu$ has atoms at both extremes with probabilities
			$p_{\min},p_{\max}$, then
			\begin{equation}\label{eq:surprisal}
				\E[T^{+}]=-\frac{\log p_{\max}}{\ell_\mu},
				\qquad
				\E[T^{-}]=-\frac{\log p_{\min}}{\ell_\mu},
			\end{equation}
			\begin{equation}\label{eq:surprisal2}
				\E[S]=\frac12+\frac{\log(p_{\min}/p_{\max})}{\ell_\mu} .
			\end{equation}
			For a sample of $n$ observations with distinct extremes,
			\begin{equation}\label{eq:samplemean}
				\E[S]=\frac12,
				\qquad
				\E|T|=\frac{2\log n}{\ell_\mu},
			\end{equation}
			whatever the interior of the sample. For $\mu=\delta_m+\delta_M$, $T$
			is logistic with scale $1/\ell_\mu$.
			
			\item If $\mu$ is finitely supported, the right tail is exponential
			with rate $\log(x_{(n)}/x_{(n-1)})$. If
			instead $\mu$ has a density $f(x)\sim c(M-x)^{k}$ as $x\uparrow M$,
			then
			\begin{equation}\label{eq:cauchytail}
				1-F_S(s)\sim\frac{k+1}{\ell_\mu s},
			\end{equation}
			and $\E[S^{+}]=\infty$. The mirror statements hold at $m$. For
			any member of the family the same dichotomy holds with the
			constant $f_Z(L_+)L_+(k+1)$ in place of $(k+1)/\ell_\mu$.
		\end{enumerate}
	\end{theorem}
	
	Assertion (ii) explains the shape of the law in one sentence, since
	the suddency moment is the Fisher information profile of the escort
	family, read as a density, blurred by one unit of suddency. The unit
	blur is the passage from the derivative $K_\mu'$ to the unit
	increment $\Lambda_\mu$ that organizes the theory of the transform
	(Figure~\ref{fig:2}(b)). A general member sees the same profile
	through the convex map $G\circ\exp$ on the logarithmic scale, and Jensen's inequality, applied in the Appendix, gives
	\begin{equation}\label{eq:jensen}
		S\ge_{\st}\widetilde T+U,
		\qquad
		F_{\widetilde T}=C\bigl(G(e^{K_\nu'})-e^{\beta}\bigr).
	\end{equation}
	Assertion (iv) says that continuous populations have suddency laws
	without a mean while samples have exponential tails, so the law of a
	sample and the law of the population it is drawn from differ in
	their tails at every sample size, and this decides the limit
	theorems.
	
	\begin{theorem}[Limit theorems]\label{thm:limits}
		Let $S_1,S_2,\dots$ be independent copies of $S$, $\bar S_n$ their
		mean and $S_{(n)}$ their maximum.
		\begin{enumerate}[label=\textup{(\roman*)},leftmargin=2.2em]
			\item Let $\mu$ be finitely supported. Then
			\begin{equation}\label{eq:clt}
				\sqrt n\,(\bar S_n-\E S)\xrightarrow{d}\mathcal N(0,\Var S),
			\end{equation}
			\begin{equation}\label{eq:gumbelmax}
				r_+S_{(n)}-\log(nc_+)\xrightarrow{d}\mathrm{Gumbel}.
			\end{equation}
			The mean $\bar S_n$ obeys Cram\'er's large-deviation principle
			with a rate function that is finite and strictly convex on $\R$.
			\item Let $\mu$ have a density with contact orders $k_\pm$ at the
			extremes, so that by Theorem~\ref{thm:jeffreys}(iv)
			$1-F_S(s)\sim c_+/s$ and $F_S(-s)\sim c_-/s$ for some constants
			$c_\pm>0$.
			Then $S$ is in the domain of attraction of the stable law of
			index one with skewness $(c_+-c_-)/(c_++c_-)$ and scale
			$\tfrac\pi2(c_++c_-)$. With the centering $b_n=\E[S;|S|\le n]$, which differs from
			$(c_+-c_-)\log n$ by a bounded quantity, $\bar S_n-b_n$ converges
			in law to it. For the Jeffreys law with
			$k_+=k_-=k$ the limit is Cauchy with scale $\pi(k+1)/\ell_\mu$.
			The maximum satisfies
			\begin{equation}\label{eq:frechet}
				\frac{S_{(n)}}{n}\xrightarrow{d}\mathrm{Fr\acute echet}(1)\ \text{with scale }c_+ .
			\end{equation}
		\end{enumerate}
	\end{theorem}
	
	\section{Applications}\label{sec:estimation}
	
	The Lehmer--Lambert distribution of Definition~\ref{def:ll} is a
	representation of a signal, not a model of the mechanism that
	produced it. The signal is the measure $\mu$ and the suddency moment
	$S$ is never observed, so the parameters $\alpha$ and $\beta$ cannot
	be estimated from a signal by likelihood. They are chosen by the
	analyst, as the order of a norm or the exponent of a power mean is
	chosen, the Jeffreys law and the canonical Omega law of
	Proposition~\ref{prop:limits} being the parameter-free choices, and
	by Proposition~\ref{prop:limits}(iv) they act on the law mainly
	through the combination $\lambda$ of definition \eqref{eq:lambdaa},
	so one of them may be fixed.
	
	The intended use of the law is the analysis of nonstationary
	signals, and we illustrate it on electroencephalograms.
	Electroencephalograms alternate between quiet stretches and bursts,
	and the bursts differ from the background in how the energy is
	spread across scales and in how isolated the largest excursions
	are. These properties, rather than mean levels, carry the
	electrophysiological markers of depression, which have been sought
	in regional band powers, hemispheric asymmetries and nonlinear
	features~\cite{henriques1991,heller1993,mumtaz2017,delatorre2017},
	and the modes of the Lehmer--Lambert law lie at the scale
	crossovers while its tails read the isolation of the extremes.
	Financial returns show the same phenomenology, with heavy tails and
	volatility clustering among their stylized
	facts~\cite{cont2001,ataei2021,ataei2025a,ataei2020,ataei2025b}.
	The recordings come from the public dataset of
	Mumtaz~\cite{mumtaz2016}, which Mumtaz et al.~\cite{mumtaz2017}
	describe, and were made during the task condition of the dataset,
	with $19$ scalp channels of the $10$--$20$ system sampled at $256$
	Hz for about ten minutes. Each channel was band-pass filtered
	between $0.5$ and $45$ Hz, without other artefact rejection, and the
	positive signal of a channel is the root-mean-square amplitude of
	its half-second epochs, about $1200$ values per channel. This
	windowed amplitude is bounded away from zero and its extremes are
	averages rather than single samples. The choice matters, because by
	Proposition~\ref{prop:invariance} the law depends on the signal
	through its log-profile, which is anchored at the two extremes,
	whereas on the raw samples the minimum is set by the quantization
	step and the maximum by isolated artefacts. Each channel has its own
	standardized curve, its own law and its own features, and no
	estimation is involved. The parameters are fixed at
	$(\alpha,\beta,L_+)=(\tfrac12,2,e)$, the level being that of
	Figure~\ref{fig:1}. The suddency values quoted below are in
	standardized units, in which the logarithmic range of a channel is
	one; the corresponding orders of the raw signal are $\tau s$ with
	$\tau=1/\ell_\mu$.
	
	\begin{figure}[p]
		\centering
		\includegraphics[width=.72\textwidth]{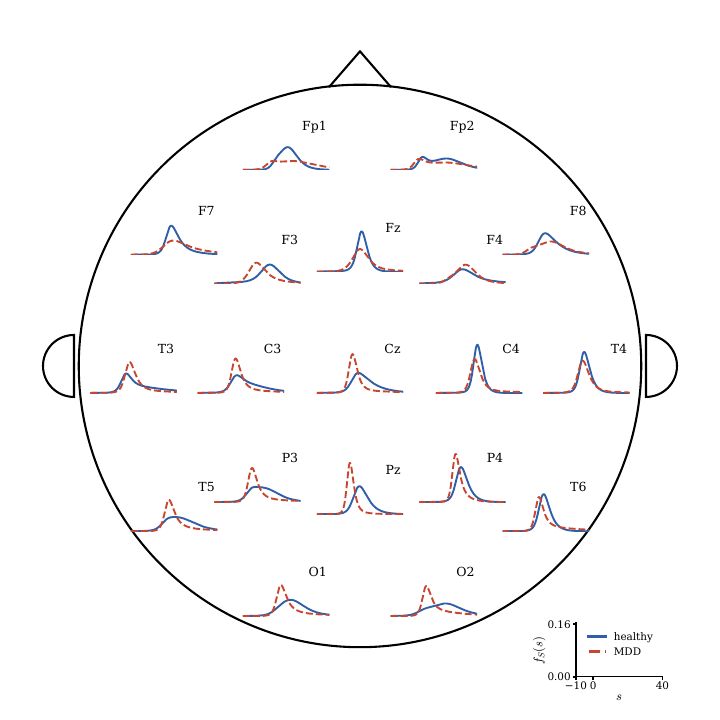}\\[2pt]
		{\small (a)}\\[8pt]
		\includegraphics[width=\textwidth]{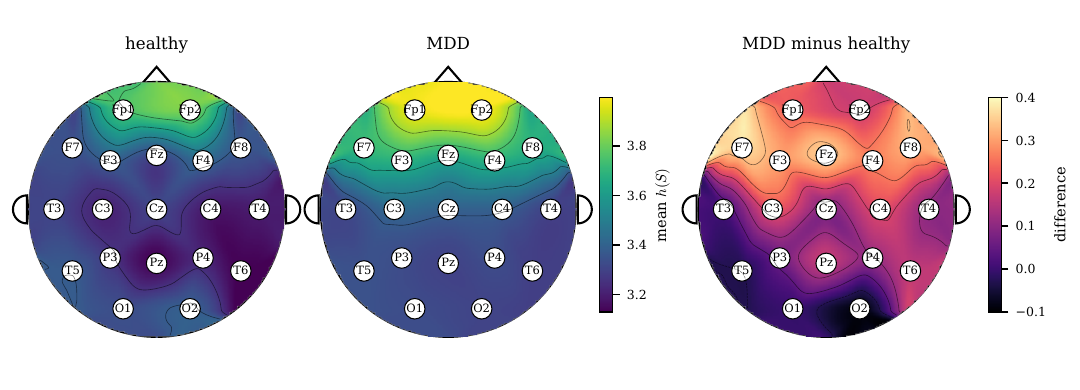}\\[2pt]
		{\small (b)}
		\caption{(a) Lehmer--Lambert densities $\LL(\tfrac12,2)$ at
			$L_+=e$ of the half-second amplitudes of the $19$ channels of a
			healthy subject (solid) and of a patient with major depressive
			disorder (dashed). (b) Entropy $h(S)$ of the same law averaged
			over the healthy subjects and over the patients, and the
			difference of the two averages.}
		\label{fig:3}
	\end{figure}
	
	Figure~\ref{fig:3}(a) shows the nineteen densities of one healthy
	subject and of one patient on the layout of the electrode cap. The
	laws are unimodal, with modes between $s=9$ and $s=22$, except at
	Fp2 of both subjects and at Fp1 of the patient, where they are
	bimodal with one mode near $s=7$ and one above $s=18$. Over the
	prefrontal electrodes and at Fz the law of the patient is broader,
	at F7, F4 and F8 the two laws nearly coincide, and at F3 and over
	the central, parietal, temporal and occipital electrodes the law of
	the patient is the sharper one, with its mode at lower suddency,
	between $s=9$ and $s=13$ against $s=13$ and $s=22$ for the healthy
	subject. By the entropy identity \eqref{eq:entropy}, a broader law
	is a signal whose tilted ensembles are more log-dispersed, that is,
	whose amplitudes occupy more scales and cross from one scale to the
	next more gradually, and the entropy of the patient averaged over
	the frontal electrodes exceeds that of the healthy subject by about
	half a nat while the two posterior averages coincide.
	
	Figure~\ref{fig:3}(b) carries the comparison to the two groups. The
	entropy of the law was computed exactly by the quadrature
	\eqref{eq:featuresZ2} for every channel of every subject, and the
	figure shows the group averages and their difference on the scalp.
	The entropy averaged over the seven frontal electrodes Fp1, Fp2, F7,
	F3, Fz, F4 and F8 is $3.48$ in the healthy group and $3.77$ in the
	patients, whereas over the seven posterior electrodes P3, Pz, P4,
	O1, O2, T5 and T6 it is $3.26$ against $3.32$. The excess in the
	patients is present at every frontal electrode, reaches $0.31$ to
	$0.36$ at F7, Fz and F8, decreases over the central electrodes and
	vanishes over the occipital and posterior temporal electrodes.
	
	The frontal pattern is consistent with the electrophysiological
	literature on depression, which locates its best established
	markers over the frontal region. Henriques and
	Davidson~\cite{henriques1991} described a frontal alpha asymmetry,
	and the analysis of the present dataset by Mumtaz et
	al.~\cite{mumtaz2017} rests on regional band powers and alpha
	asymmetry. The model of Heller~\cite{heller1993} adds a reduced
	activity of the right parietotemporal region, but the entropy of the law shows no
	posterior difference and no reliable left--right asymmetry. The
	irregularity of the electroencephalogram, measured by entropy rates,
	fractal dimensions and Lempel--Ziv complexity, is reported to be
	elevated in depression over large cortical
	areas~\cite{li2008,delatorre2017}, and the higher frontal entropy
	found here points in the same direction, although $h(S)$ measures
	not temporal irregularity but the number of amplitude scales that
	the tilted signal passes through. The modes and the entropy that
	Ataei and Wang~\cite{ataei2022} found discriminative on the
	resting-state recordings of the same dataset are thus given a
	meaning by the present theory.
	
	\section{Conclusion}\label{sec:conclusion}
	
	The Lehmer--Lambert distribution is the law obtained by composing
	the Lehmer transform of a positive signal with the Lehmer curve of
	an exponential--lognormal product. The composition is explicit at
	every step, and the law reads a signal through its log-profile, so
	it is invariant under changes of gain and of power-law calibration.
	Its modes lie at the scale crossovers of the tilted signal, its tails read the two extreme logarithmic gaps of a sample, its entropy is the mean
	logarithmic dispersion under tilting, and its logarithmic limit is a
	parameter-free law whose density is a Jeffreys divergence. The
	random Lehmer mean carries the parameters and has Kummer moments
	with a first-order recurrence, while the suddency moment carries the
	signal and has a moment generating function whose domain and boundary
	behaviour are decided by the extreme gaps alone. When the law represents a signal its parameters are chosen rather
	than estimated. On electroencephalograms of healthy subjects and of patients with
	major depressive disorder, the entropy of the law is higher in the
	patients over the frontal region and unchanged over the posterior
	region, a pattern consistent with the frontal location of the
	electrophysiological markers of depression.
	
	Several questions remain. The number of modes of the law can
	exceed the number of concavity episodes of the curve, and which
	signals and parameters give exactly one mode per episode is open. The infinite
	divisibility of the two-point Jeffreys law, the sum of a logistic
	and a uniform variable, is also open. The composition principle
	raises a moment problem, namely which increasing analytic functions
	on a bounded interval are Lehmer curves. On the applied side, the
	quantiles, the tail rates and the moments of the random Lehmer mean
	are further features of a signal, and the same programme applies to financial series, whose heavy
	tails and volatility clustering are the properties the law reads.
	
	\bmhead{Data Availability}
	No new data were generated in this work.
	
	\bmhead{Conflict of Interest}
	The authors declare no conflict of interest.
	
	\bmhead{Funding Declaration}
	This research has received no funding.

	\clearpage
	\begin{appendices}
		\section{Proofs}\label{app:proofs}
		
		Throughout the appendix,
		$a=1/\alpha$, so that
		\begin{equation}\label{eq:Gprime}
			G(z)=z^{a}e^{\beta z},
			\qquad
			G'(z)=\Bigl(\frac az+\beta\Bigr)G(z).
		\end{equation}
		
		\begin{proof}[Proof of Theorem~\ref{thm:origin} (Composition of two curves)]
			By Theorem~\ref{thm:calculus}(iii),(iv),
			$\Leh_{E\Lambda}=\Leh_E\Leh_\Lambda$ with $\Leh_E(s)=s$ and
			$\Leh_\Lambda(s)=e^{m_0+\sigma^{2}(s-1/2)}$, $m_0=\sigma^{2}/2$.
			This gives $se^{\sigma^{2}s}$. Solving $se^{\sigma^{2}s}=t$ as
			\begin{equation}\label{eq:lambertsolve}
				(\sigma^{2}s)\,e^{\sigma^{2}s}=\sigma^{2}t
			\end{equation}
			gives the inverse on the principal branch, since $\sigma^{2}s>0$.
			The variable $\log E$ is a reflected Gumbel variable, which is
			infinitely divisible by the classical results of Steutel and van
			Harn~\cite{steutel2004}, and $\log\Lambda$ is
			Gaussian. So $\log(E\Lambda)$ is infinitely divisible, and the
			convolution powers define $\rho_t$ with $Z_{\rho_t}=Z_\rho^{\,t}$,
			hence $\Leh_{\rho_t}=\Leh_\rho^{\,t}$. Then
			\begin{equation}\label{eq:Gpower}
				\Leh_\rho(z)^{t}=z^{t}e^{t\sigma^{2}z}=z^{1/\alpha}e^{\beta z}
			\end{equation}
			for $t=1/\alpha$ and $t\sigma^{2}=\beta$, so that the distribution
			function \eqref{eq:llF} takes the form \eqref{eq:composition}. Finally $F_S(S)$ is uniform, so
			$V\eqdef G(\Leh_\nu(S))=F_S(S)/C+e^{\beta}$ is uniform on
			$(G(1),G(L_+))$ and $S=\Leh_\nu^{-1}(G^{-1}(V))$.
		\end{proof}
		
		\begin{proof}[Proof of Theorem~\ref{thm:basic} (Density and quantiles)]\leavevmode
			
			\noindent(i) By formula \eqref{eq:Gprime}, $G'>0$ on $(1,L_+)$. So the distribution
			function $C(G-e^{\beta})$ of $Z$ has density $CG'$, which is
			the density \eqref{eq:lgdensity}.
			
			\noindent(ii) By the chain rule
			$f_S=f_Z(\Leh_\nu)\Leh_\nu'$, and the covariance identity \eqref{eq:covid} identifies
			$\Leh_\nu'(s)$ with the escort covariance. The power-sum form is
			formula \eqref{eq:powersumslope} applied to the standardized data.
			
			\noindent(iii) Set
			$F_S(s)=p$ with $z=\Leh_\nu(s)$. Then $z^{a}e^{\beta z}=p/C+e^{\beta}$.
			Raising to the power $\alpha$ and multiplying by $\alpha\beta$
			gives
			\begin{equation}\label{eq:quantilesolve}
				(\alpha\beta z)\,e^{\alpha\beta z}=\alpha\beta\bigl(p/C+e^{\beta}\bigr)^{\alpha},
			\end{equation}
			whence $\alpha\beta z=W_0(\cdot)$ on the principal branch, since
			$\alpha\beta z>0$; at $\beta=0$ the equation $z^{a}=p/C+1$ is solved
			directly.
			
			\noindent(iv) $G(Z)=F_S(S)/C+e^{\beta}$ is an affine image
			of the uniform variable $F_S(S)$. The converse is (iii) at $p=U$.
			
			\noindent(v) is immediate from $1-F_S=C(A-G(\Leh_\nu))$.
		\end{proof}
		
		\begin{proof}[Proof of Proposition~\ref{prop:invariance} (Invariance)]
			The power map is the composition of $x\mapsto x/m$ and
			$y\mapsto y^{\tau}$. By Theorem~\ref{thm:calculus}(i),(ii),
			$Z_\nu(s)=m^{-\tau s}Z_\mu(\tau s)$. Since
			\begin{equation}\label{eq:thetaexp}
				(x/m)^{\tau s}=\exp\bigl(\tau s\ell_\mu\theta(x)\bigr)=e^{\ell s\theta(x)},
			\end{equation}
			also $Z_\nu(s)=\mu(\Rpos)\int_0^1e^{\ell s\theta}\dd\pi(\theta)$.
			The logarithm of the ratio at consecutive orders is
			representation \eqref{eq:profilecurve}, the total mass cancelling, and
			$y_i=(x_i/m)^{\tau}=L_+^{\theta_i}$. The map $x\mapsto cx^{r}$ sends
			$\log x$ to $\log c+r\log x$ and leaves $\theta$ unchanged.
			Conversely two measures with the same log-profile are related by
			such a map, since the profile determines $\mu/\mu(\Rpos)$ up to the
			affine map of the logarithmic scale sending $[\log m,\log M]$ onto
			$[0,1]$. The reflection sends $\theta$ to $1-\theta$ and the
			standardized atoms to $L_+/y_i$. So $Z_{\check\nu}(s)=L_+^{s}Z_\nu(-s)$
			and
			\begin{equation}\label{eq:reflectionproof}
				\Leh_{\check\nu}(s)=\frac{L_+Z_\nu(-s)}{Z_\nu(1-s)}=\frac{L_+}{\Leh_\nu(1-s)} .
			\end{equation}
		\end{proof}
		
		\begin{proof}[Proof of Proposition~\ref{prop:limits} (Special cases and limits)]\leavevmode
			
			\noindent(i) is the distribution function \eqref{eq:llF} at $\beta=0$.
			
			\noindent(ii) As $\alpha\to\infty$,
			\begin{equation}\label{eq:loglimit}
				\frac{z^{1/\alpha}-1}{L_+^{1/\alpha}-1}\longrightarrow\frac{\log z}{\log L_+}.
			\end{equation}
			For the raw data the map is $(\log z-\log m)/\ell_\mu$. Since
			$F_S(S)$ is uniform, $\log\Leh_\mu(S)$ is uniform on
			$(\log m,\log M)$, and the mean of the log-uniform law is
			$(M-m)/\ell_\mu$.
			
			\noindent(iii) With $\Leh_\rho(y)=ye^{y}$ the map
			$z\mapsto(z-1)e^{z-1}$ increases from $0$ at $z=1$ to
			$\Omega e^{\Omega}=1$ at $z=1+\Omega$. So
			$F_S(s)=\Leh_\rho(\Leh_\nu(s)-1)$ is a distribution function.
			$Y=\Leh_\nu(S)-1$ has distribution function $ye^{y}$ and density
			$(1+y)e^{y}$ on $[0,\Omega]$, and $Ye^{Y}=F_S(S)=U$ inverts to
			$Y=W_0(U)$. Integrating $y^{k}$ against $\dd(ye^{y})$ by parts
			gives $\E[Y^{k}]=\Omega^{k}-k\int_0^{\Omega}y^{k}e^{y}\dd y$, a
			polynomial in $\Omega^{\pm1}$ since $e^{\Omega}=\Omega^{-1}$.
			
			\noindent(iv) Put
			$z=L_+-\tau'/\lambda$. Then
			\begin{equation}\label{eq:gumbelstep}
				\Prob\bigl(\lambda(L_+-Z)>\tau'\bigr)=C\bigl(G(z)-e^{\beta}\bigr)
				=\frac{G(z)/A-e^{\beta}/A}{1-e^{\beta}/A},
			\end{equation}
			and
			\begin{equation}\label{eq:gumbellog}
				\log\frac{G(z)}{A}=a\log\Bigl(1-\frac{\tau'}{\lambda L_+}\Bigr)-\frac{\beta\tau'}{\lambda}
				=-\tau'+O\Bigl(\frac{\tau'^{2}}{\lambda L_+}\Bigr),
			\end{equation}
			since $a\le\lambda L_+$. Also
			$e^{\beta}/A=\exp\{-\beta(L_+-1)-a\log L_+\}\to0$, being bounded by
			$\exp\{-\min(L_+-1,L_+\log L_+)\lambda\}$. Hence
			$\lambda(L_+-Z)\to\mathrm{Exp}(1)$. By the asymptotic relation \eqref{eq:georate}
			applied to $\nu$, $L_+-\Leh_\nu(s)=a_+e^{-r_+s}(1+\varepsilon(s))$ with
			$\varepsilon(s)\to0$, and $\Leh_\nu(S)=Z$ gives
			\begin{equation}\label{eq:gumbelidentity}
				r_+S-\log(a_+\lambda)=-\log\bigl(\lambda(L_+-Z)\bigr)+\log\bigl(1+\varepsilon(S)\bigr).
			\end{equation}
			Since $L_+-Z\to0$ in probability, $S\to\infty$ and
			$\varepsilon(S)\to0$ in probability. The limit is $-\log E$ with
			$E$ standard exponential, the standard Gumbel law.
			
			\noindent(v) By
			representation \eqref{eq:profilecurve} and Theorem~\ref{thm:calculus}(ii),
			\begin{equation}\label{eq:smallrange}
				\log\Leh_\nu(u/\tau)=K_\mu(u)-K_\mu(u-\tau)-\tau\log m
				=\tau\bigl(K_\mu'(u)-\log m\bigr)+O(\tau^{2})
			\end{equation}
			uniformly in $u$, because $K_\mu''\le\ell_\mu^{2}/4$. Since
			$\tau\ell_\mu=\log L_+=(L_+-1)(1+O(L_+-1))$,
			$\Leh_\nu(u/\tau)=1+(L_+-1)[F_T(u)+O(\tau)]$ with
			$F_T=(K_\mu'-\log m)/\ell_\mu$. Then
			$F_S(u/\tau)=C(G(1+(L_+-1)[F_T(u)+O(\tau)])-e^{\beta})\to F_T(u)$
			uniformly, because $C(G(1+(L_+-1)p)-e^{\beta})\to p$ uniformly in
			$p\in[0,1]$ as $L_+\downarrow1$, $G$ being $C^{1}$ with $G'(1)>0$.
		\end{proof}
		
		\begin{proof}[Proof of Theorem~\ref{thm:shape} (Shape)]\leavevmode
			
			\noindent(i) From the density \eqref{eq:lgdensity},
			\begin{equation}\label{eq:psiprime}
				\psi'(z)=-\frac{(\alpha\beta)^{2}}{(1+\alpha\beta z)^{2}}-\frac{1-\alpha}{\alpha z^{2}} .
			\end{equation}
			For $\alpha\le1$ the three terms of $\psi$ in formula \eqref{eq:psi}
			are nonnegative and the two terms of $\psi'$ are nonpositive,
			whatever $\beta\ge0$, so $f_Z$ is nondecreasing and log-concave;
			both are strict unless $\alpha=1$ and $\beta=0$, when $\psi\equiv0$.
			For $\alpha>1$ the middle term of $\psi$ equals $1/\alpha-1<0$ at
			$z=1$ and dominates when $\beta$ is small, while for large $\beta$
			the last term dominates. The hazard properties
			are the standard ones of log-concave densities on an interval, the
			divergence following from $A-G(z)\to0$. Convexity of $G$ follows
			from
			\begin{equation}\label{eq:Gsecond}
				G''=\bigl[a(a-1)z^{-2}+2a\beta z^{-1}+\beta^{2}\bigr]G,
			\end{equation}
			which is nonnegative for all $z\ge1$ and all $\beta\ge0$ if and only
			if $a\ge1$, since at $\beta=0$ it has the sign of $a-1$.
			
			\noindent(ii)
			Logarithmic differentiation of the density \eqref{eq:ll} gives
			\begin{equation}\label{eq:logfS}
				(\log f_S)'=\psi(\Leh_\nu)\Leh_\nu'+\frac{\Leh_\nu''}{\Leh_\nu'} .
			\end{equation}
			At a stationary point $\psi(\Leh_\nu)\Leh_\nu'\ge0$, since
			$\psi\ge0$ for $\alpha\le1$ and $\Leh_\nu'>0$, so $\Leh_\nu''\le0$.
			The inequality is strict unless $\psi\equiv0$, that is unless
			$\alpha=1$ and $\beta=0$, when the density of $S$ is proportional
			to $\Leh_\nu'$ and its modes are the inflection points at which
			the curve turns concave.
			
			\noindent(iii) For
			$\beta'>\beta$ the ratio $f_Z(z;\alpha,\beta')/f_Z(z;\alpha,\beta)$
			is proportional to
			\begin{equation}\label{eq:lrbeta}
				\frac{1+\alpha\beta'z}{1+\alpha\beta z}\,e^{(\beta'-\beta)z},
			\end{equation}
			a product of increasing functions. For $\alpha'<\alpha$ the log-derivative in $z$ of $f_Z(z;\alpha',\beta)/f_Z(z;\alpha,\beta)$
			equals
			\begin{equation}\label{eq:lralpha}
				\frac{\alpha-\alpha'}{\alpha\alpha'z}\Bigl[1-\frac{\alpha\alpha'\beta z}{(1+\alpha'\beta z)(1+\alpha\beta z)}\Bigr],
			\end{equation}
			whose bracket is positive when $\alpha'\le1$, because
			\begin{equation}\label{eq:lralphabracket}
				(1+\alpha'\beta z)(1+\alpha\beta z)-\alpha\alpha'\beta z
				=1+\alpha'\beta z+\alpha\beta z(1-\alpha')+\alpha\alpha'\beta^{2}z^{2}>0 .
			\end{equation}
			When $\alpha'>1$ the third term is negative and the left side of
			identity \eqref{eq:lralphabracket} is negative on an interval of
			values of $\beta z$ as soon as
			$(\alpha+\alpha'-\alpha\alpha')^{2}>4\alpha\alpha'$, so the
			restriction cannot be dropped. Both laws
			of $S$ are images of laws on $(1,L_+)$ under the common bijection
			$\Leh_\nu^{-1}$. So the likelihood-ratio order transfers to the
			suddency line, and it implies the usual stochastic order, hence
			the monotonicity of quantiles and means. The Jeffreys law of $\nu$
			makes $\log Z$ uniform on $(0,\ell)$ by
			Proposition~\ref{prop:limits}(ii). So the likelihood ratio of $S$
			with respect to $S_J$ is $\ell zf_Z(z)$ at $z=\Leh_\nu(s)$, which
			is $\ell C(a+\beta z)z^{a}e^{\beta z}$, increasing in $z$ and hence
			in $s$. Stochastic order gives $\E[S]\ge\E[S_J]=\tfrac12$ by
			Theorem~\ref{thm:jeffreys}(iii) applied to $\nu$.
			
			\noindent(iv) is the mean value theorem applied to the map
			$z\mapsto C(G(z)-e^{\beta})$, whose derivative is $f_Z$; for
			$\alpha\le1$ the density is nondecreasing by (i), so its supremum is
			$f_Z(L_+)=CG'(L_+)$.
		\end{proof}
		
		\begin{proof}[Proof of Theorem~\ref{thm:tails} (Tails and generating function)]\leavevmode
			
			\noindent(i) Write $y_1<\dots<y_k=L_+$ for the distinct standardized atoms,
			$w_1,\dots,w_k$ for their masses, and $\kappa_i=\log(L_+/y_i)$, so
			that $0=\kappa_k<\kappa_{k-1}=r_+<\dots$. Then
			\begin{equation}\label{eq:curveN}
				\Leh_\nu(s)=L_+\frac{N(s)}{N(s-1)},
				\qquad
				N(s)=w_k+\sum_{i<k}w_ie^{-\kappa_is}.
			\end{equation}
			For $s$ beyond some $s_1$ the reciprocal $1/N(s-1)$ is a convergent
			geometric series in $(N(s-1)-w_k)/w_k$. Each power is a finite exponential polynomial with exponents in the
			additive semigroup $\Lambda_+$ generated by $\kappa_1,\dots,\kappa_{k-1}$.
			The semigroup is discrete, since an element with $n$ summands is
			at least $nr_+$, so its elements can be listed in increasing order
			and
			\begin{equation}\label{eq:curvedirichlet}
				L_+-\Leh_\nu(s)=\sum_{j\ge1}a_je^{-\kappa^{(j)}s},
				\qquad s\ge s_1,
			\end{equation}
			absolutely convergent with exponents $\kappa^{(1)}<\kappa^{(2)}<\dots$
			in $\Lambda_+$, and
			\begin{equation}\label{eq:a1}
				a_1=L_+\frac{w_{k-1}}{w_k}\bigl(e^{r_+}-1\bigr).
			\end{equation}
			Since $G$ is analytic at $L_+$ and
			$\delta(s)\eqdef L_+-\Leh_\nu(s)\to0$,
			\begin{equation}\label{eq:tailtaylor}
				1-F_S(s)=C\bigl(G(L_+)-G(L_+-\delta(s))\bigr)
				=\sum_{j\ge1}(-1)^{j+1}\frac{CG^{(j)}(L_+)}{j!}\delta(s)^{j}
			\end{equation}
			converges for $s$ large. Powers of the series
			in expansion \eqref{eq:curvedirichlet} are again such series. Rearranging, for
			$s$ beyond some $s_2$,
			\begin{equation}\label{eq:taildirichlet}
				1-F_S(s)=\sum_{j\ge1}c_je^{-\kappa^{(j)}s},
				\qquad c_1=CG'(L_+)a_1=c_+ ,
			\end{equation}
			which is the first of the tail relations \eqref{eq:tails} with
			$\delta=\kappa^{(2)}-\kappa^{(1)}$. Term-by-term differentiation
			gives $f_S(s)\sim r_+c_+e^{-r_+s}$ and the hazard limit. The
			quantile statement inverts the tail. The left tail is the same
			computation at the other end, where
			$\Leh_\nu(s)-1\sim(w_m'/w_m)(1-e^{-r_-})e^{r_-s}$ as $s\to-\infty$,
			and the expansion proceeds in the exponentials $e^{s\log y_i}$.
			
			\noindent(ii) Finiteness of $M_S$ on $(-r_-,r_+)$ and divergence outside are
			immediate from the tail relations \eqref{eq:tails}. For $-r_-<t<r_+$ integration by
			parts gives
			\begin{equation}\label{eq:mgfparts}
				M_S(t)=1+t\int_0^{\infty}e^{ts}\bigl(1-F_S(s)\bigr)\dd s
				-t\int_{-\infty}^{0}e^{ts}F_S(s)\dd s .
			\end{equation}
			Inserting expansion \eqref{eq:taildirichlet} beyond $s_2$,
			\begin{equation}\label{eq:mgfseries}
				t\int_{s_2}^{\infty}e^{ts}\bigl(1-F_S(s)\bigr)\dd s
				=t\sum_{j\ge1}c_j\frac{e^{(t-\kappa^{(j)})s_2}}{\kappa^{(j)}-t},
			\end{equation}
			the interchange being justified by
			$\sum_j|c_j|\int_{s_2}^\infty e^{(t-\kappa^{(j)})s}\dd s
			\le(r_+-t)^{-1}e^{ts_2}\sum_j|c_j|e^{-\kappa^{(j)}s_2}<\infty$.
			The $j=1$ term is $r_+c_+/(r_+-t)+O(1)$ as $t\uparrow r_+$ and the
			rest stays bounded. This is the first of the Abelian relations
			\eqref{eq:abelian}. The second is symmetric. Steepness follows from
			$M_S'(t)=r_+c_+/(r_+-t)^{2}+O((r_+-t)^{-1})$.
			
		\end{proof}
		
		\begin{proof}[Proof of Theorem~\ref{thm:moments} (Moments and entropy)]\leavevmode
			
			\noindent(i) Since $f_Z=CG'$, integration by parts gives, for complex $w$,
			\begin{equation}\label{eq:momentparts}
				\E[Z^{w}]=C\bigl[z^{w}G\bigr]_1^{L_+}-wC\int_1^{L_+}z^{w-1}G(z)\dd z,
			\end{equation}
			and $z^{w-1}G(z)=z^{w+a-1}e^{\beta z}$ is the integrand of
			$M_\beta(w+a)$, entire in $w$. The same computation with $e^{tz}$
			gives formula \eqref{eq:lgmgf}, the integrand $e^{tz}G(z)=z^{a}e^{(\beta+t)z}$
			being that of $M_{\beta+t}(a+1)$.
			
			\noindent(ii) Directly,
			\begin{equation}\label{eq:momentdirect}
				\E[Z^{k}]=C\int_1^{L_+}z^{k}G'(z)\dd z
				=C\bigl[aM_\beta(k+a)+\beta M_\beta(k+a+1)\bigr].
			\end{equation}
			Equating with $H_k-kCM_\beta(k+a)$ from (i) gives the recurrence
			\begin{equation}\label{eq:Mrec}
				\beta M_\beta(k+a+1)=\bigl(L_+^{k}A-e^{\beta}\bigr)-(k+a)M_\beta(k+a),
			\end{equation}
			whose case $k=0$ is the normalization
			$1=C[aM_\beta(a)+\beta M_\beta(a+1)]$. Eliminating $M_\beta(1+a)$
			from $m_1=H_1-CM_\beta(1+a)$ gives formula \eqref{eq:m1}. For $k\ge2$, (i)
			gives $CM_\beta(k-1+a)=(H_{k-1}-m_{k-1})/(k-1)$ and recurrence \eqref{eq:Mrec}
			at $k-1$ gives
			$CM_\beta(k+a)=\beta^{-1}[H_{k-1}-(k-1+a)CM_\beta(k-1+a)]$.
			Substituting into $m_k=H_k-kCM_\beta(k+a)$ yields
			recurrence \eqref{eq:momentrec}.
			
			\noindent(iii) $h(Z)=-\E\log f_Z(Z)$ is
			read off the density \eqref{eq:lgdensity}. By the density
			formula \eqref{eq:ll},
			$h(S)=-\E\log f_Z(Z)-\E\log\Leh_\nu'(S)$ with
			$\Leh_\nu'=\Cov_{\nu_{s-1}}(X,\log X)$.
		\end{proof}
		
		\begin{proof}[Proof of Theorem~\ref{thm:jeffreys} (Jeffreys law)]\leavevmode
			
			\noindent(i) By relation \eqref{eq:jeffreysslope},
			$\Lambda_\mu'(s)=J(\mu_{s-1},\mu_s)=\int_{s-1}^{s}I$, and dividing
			by $\ell_\mu$ gives the density. $K_\mu'$ increases from $\log m$
			to $\log M$ by Theorem~\ref{thm:statgen} through the master
			identity, so $\int_\R I=\ell_\mu$.
			
			\noindent(ii) $F_T$ increases from $0$ to
			$1$ with density $I/\ell_\mu$, and
			\begin{equation}\label{eq:convolution}
				f_S(s)=\int_{s-1}^{s}f_T(u)\dd u=\bigl(f_T*\mathbf 1_{[0,1]}\bigr)(s)
			\end{equation}
			is the density of $T+U$. Differentiating $F_S=\int_{s-1}^{s}F_T$
			gives $f_S=F_T(s)-F_T(s-1)\le1$, and $f_S'=(I(s)-I(s-1))/\ell_\mu$.
			Cumulants add, and the cumulant function of $U$ is
			\begin{equation}\label{eq:uniformcumulants}
				\log\frac{e^{t}-1}{t}=\sum_{j\ge1}\frac{B_j}{j}\frac{t^{j}}{j!} .
			\end{equation}
			
			\noindent(iii) We have
			\begin{equation}\label{eq:Tplus}
				\E[T^{+}]=\int_0^\infty\bigl(1-F_T\bigr)
				=\frac1{\ell_\mu}\Bigl[K_\mu(0)-\lim_{u\to\infty}\bigl(K_\mu(u)-u\log M\bigr)\Bigr],
			\end{equation}
			and $Z_\mu(u)M^{-u}=\int(x/M)^{u}\dd\mu\to w_M$ by dominated
			convergence. So $\E[T^{+}]=(\log\mu(\Rpos)-\log w_M)/\ell_\mu=-\log p_{\max}/\ell_\mu$.
			$\E[T^{-}]$ is the same computation at $-\infty$, and
			$\E[S]=\E[T]+\tfrac12$ by (ii). For a sample with distinct extremes
			$p_{\min}=p_{\max}=1/n$.
			
			For two atoms $\mu_u$ is Bernoulli with success probability
			$\varsigma(\ell_\mu u)$ on $M$, so $F_T(u)=\varsigma(\ell_\mu u)$ is
			logistic with scale $1/\ell_\mu$.
			
			\noindent(iv) The atomic case is
			the asymptotic relation \eqref{eq:georate} with the positivity of the map's derivative at
			the endpoint. For the continuous case write
			$M-\Leh_\mu(s)=\E_{\mu_{s-1}}[M-X]$. Under $\mu_u$ the density of
			$M-X$ near $0$ is proportional to $(1-t/M)^{u}t^{k}\sim e^{-ut/M}t^{k}$,
			and Watson's lemma gives
			\begin{equation}\label{eq:watson}
				\E_{\mu_u}[M-X]\sim\frac{(k+1)M}{u},
			\end{equation}
			the contribution of $\{x\le M-\varepsilon\}$ being exponentially
			smaller. Hence $1-F_S(s)=\Phi(M)-\Phi(\Leh_\mu(s))\sim\Phi'(M)(k+1)M/s$
			for any map $\Phi$ with continuous positive derivative at $M$. This
			is $(k+1)/(\ell_\mu s)$ for $\Phi=\log(\cdot/m)/\ell_\mu$ and
			$f_Z(L_+)L_+(k+1)/s$ on the standardized range, the power map
			preserving the order of contact. The mirror statements are the
			same computation at $m$.
		\end{proof}
		
		\begin{proof}[Proof of the bound \eqref{eq:jensen}]
			By relation \eqref{eq:jeffreysslope}, $\Lambda_\nu(s)=\int_{s-1}^{s}K_\nu'(u)\dd u$
			is the average of $K_\nu'$ over an interval of length one. The map
			$x\mapsto G(e^{x})=\exp(ax+\beta e^{x})$ is convex, being the
			exponential of a convex function, so Jensen's inequality gives
			$G(\Leh_\nu(s))\le\int_{s-1}^{s}G(e^{K_\nu'(u)})\dd u$, hence
			$F_S(s)\le\int_{s-1}^{s}F_{\widetilde T}(u)\dd u=\Prob(\widetilde T+U\le s)$,
			which is the stochastic order stated.
		\end{proof}
		
		\begin{proof}[Proof of Theorem~\ref{thm:limits} (Limit theorems)]\leavevmode
			
			\noindent(i) The exponential tails of Theorem~\ref{thm:tails} give all
			moments and the central limit theorem, and
			$n(1-F_S)((x+\log(nc_+))/r_+)\to e^{-x}$ gives the Gumbel limit of
			the maximum. Cram\'er's theorem, in the form given by Dembo and
			Zeitouni~\cite{dembo1998} as Theorem 2.2.3, gives the large
			deviation principle with rate the Legendre transform $I_S$ of
			$K_S$. Since $K_S$ is finite, differentiable and strictly convex
			on $(-r_-,r_+)$ and steep by Theorem~\ref{thm:tails}(ii), Lemma
			2.2.5 and Lemma 2.2.20 of Dembo and Zeitouni~\cite{dembo1998} give
			that $I_S$ is finite and strictly convex on $\R$.
			
			\noindent(ii) A law
			with tails $c_\pm/s$ is in the domain of attraction of the stable
			law of index one with the stated skewness and scale, with norming
			$a_n=n$ and centering $b_n=\E[S;|S|\le n]$, by the classical
			criterion of regular variation and tail balance, see Feller~\cite{feller1971}. The norming is
			$a_n=n$ rather than $n(c_++c_-)$ because the slowly varying factor
			of the tails is a constant, and the constants $c_\pm$ are then
			carried by the scale of the limit. $\E[S;|S|\le n]$ differs from
			$(c_+-c_-)\log n$ by a bounded amount and converges when $c_+=c_-$.
			In that case the limit is Cauchy with scale $\tfrac\pi2(c_++c_-)$,
			which is $\pi(k+1)/\ell_\mu$ for the Jeffreys law with $k_\pm=k$.
			The maximum follows from $n(1-F_S(nx))\to c_+/x$ by the
			extreme-value criterion of Resnick~\cite{resnick2007}.
		\end{proof}
		
	\end{appendices}
	
\end{document}